\documentclass[preprint,1p,10pt]{elsarticle}

\usepackage{lineno,hyperref}
\modulolinenumbers[1]

\usepackage{array}
\newcolumntype{C}[1]{>{\centering\arraybackslash}m{#1}}

\usepackage{graphics}
\usepackage{amssymb, latexsym, amsmath, epsfig}
\usepackage{graphicx}
\usepackage{color}
\usepackage{epstopdf}
\usepackage{soul}
\usepackage[normalem]{ulem}
\usepackage{float}
\usepackage{amsfonts,amstext,amsthm}
\usepackage{epsfig}
\usepackage{subfigure}

\graphicspath{{figures/}}

\newtheorem{theorem}{Theorem}
\newtheorem{lemma}{Lemma}

\numberwithin{equation}{section}
\newtheorem{remark}{Remark}
\newcommand{\bfs}[1]{{\boldsymbol #1}}

\biboptions{numbers,sort&compress}

\journal{arXiv}

\begin{document}

\begin{frontmatter}

\title{Optimal spectral approximation of $2n$-order differential operators by mixed isogeometric analysis}

\author[ad,ad1]{Quanling Deng\corref{corr}}
\cortext[corr]{Corresponding author}
\ead{Quanling.Deng@curtin.edu.au}

\author[ad,ad1]{Vladimir Puzyrev}
\ead{Vladimir.Puzyrev@curtin.edu.au}

\author[ad,ad1,ad2]{Victor Calo}
\ead{Victor.Calo@curtin.edu.au}

\address[ad]{Applied Geology, Curtin University, Kent Street, Bentley, Perth, WA 6102, Australia}
\address[ad1]{Curtin Institute for Computation, Curtin University, Kent Street, Bentley, Perth, WA 6102, Australia}
\address[ad2]{Mineral Resources, Commonwealth Scientific and Industrial Research Organisation (CSIRO), Kensington, Perth, WA 6152, Australia}

\begin{abstract}
We approximate the spectra of a class of $2n$-order differential operators using isogeometric analysis in mixed formulations. This class includes a wide range of differential operators such as those arising in elliptic, biharmonic, Cahn-Hilliard, Swift-Hohenberg, and phase-field crystal equations. The spectra of the differential operators are approximated by solving differential eigenvalue problems in mixed formulations, which require auxiliary parameters. The mixed isogeometric formulation when applying classical quadrature rules leads to an eigenvalue error convergence of order $2p$ where $p$ is the order of the underlying B-spline space. We improve this order to be $2p+2$ by applying optimally-blended quadrature rules developed in \cite{puzyrev2017dispersion,calo2017dispersion} and this order is an optimum in the view of dispersion error. We also compare these results with the mixed finite elements and show numerically that mixed isogeometric analysis leads to significantly better spectral approximations.

\end{abstract}


\end{frontmatter}


\section{Introduction} \label{sec:intr} 

The finite element method (FEM) is a widely used and highly effective numerical technique for approximate solutions of boundary value problems. The theory of FEM has been extensively developed during the last 60 years. Nowadays, many different variants of FEM are used for the solutions of various complex linear and nonlinear problems. A special group of FEM techniques is the \textit{mixed finite element methods} \cite{brezzi1974existence, malkus1978mixed, brezzi1991mixed, auricchio2004mixed, gatica2014simple}. The word ``mixed'' in this case indicates that there are extra functional spaces used to approximate different solution variables. Traditionally these could be the function and its gradient where each is approximated with a different discrete representation. The literature on mixed finite element methods is quite vast and ranges from the first studies in the 1970s by Brezzi \cite{brezzi1974existence}, Babu\v{s}ka \cite{babuvska1973finite}, Crouzeix and Raviart \cite{crouzeix1973conforming} to recent contributions \cite{cervera2010mixed, gatica2014simple, wang2014weak, chung2015mixed, john2017divergence}. A large amount of research has been devoted to various stabilization techniques for the mixed methods \cite{franca1988two, douglas1989absolutely, kechkar1992analysis, masud2002stabilized, dohrmann2004stabilized, cervera2010mixed} as well as their error estimates \cite{falk1980error, arnold1985mixed}.

The mixed formulation is naturally used for problems with two independent variables such as velocity and pressure in the Stokes equations. In other popular cases, the second variable is the first or the second derivative of the original variable and approximating it directly in the problem formulation has physical interest. For example, in elasticity, both the stress and the displacement can enter the formulation at the same time. Another example is the Darcy flow equation where the mixed variational formulation is posed in terms of the function spaces $L^2(\Omega) / \mathbb{R}$ and $\mathbf{H}(div,\Omega)$ for the pressure and velocity, respectively (though other options are possible as well, see, e.g., \cite{masud2002stabilized}). The mixed finite element framework allows to preserve mass conservation, a property that is important in fluid flow problems \cite{donea1982arbitrary, masud2002stabilized, hughes2005conservation} and makes the mixed methods competitive numerical techniques in many engineering applications. The mixed problem discretization leads to a linear algebraic system of the saddle point form
\begin{equation} \label{eq:mixed_system}
\begin{bmatrix}
  \mathbf{A} & \mathbf{B}^T \\
  \mathbf{B} & 0 \\
\end{bmatrix}
\begin{bmatrix}
  \mathbf{x} \\
  \mathbf{y} \\
\end{bmatrix}
=\\
\begin{bmatrix}
  \mathbf{f} \\
  \mathbf{g} \\
\end{bmatrix}.
\end{equation}

The mixed isogeometric analysis is a relatively unexplored topic. Again, most of the work in this area focuses on the fluid flow problems where various isogeometric formulations were applied to the Stokes problem \cite{buffa2011isogeometric, evans2013isogeometric, hoang2017mixed, sarmiento2017petiga}. Recent advances include mixed isogeometric formulations for elasticity \cite{duddu2012finite} and poromechanics \cite{dortdivanlioglu2018mixed, bekele2017mixed}. In this field, coupling the fluid pressure and the solid deformation, mixed isogeometric formulations violate the inf-sup condition and suffer from numerical instabilities in the incompressible and the nearly incompressible limit. To overcome the numerical instabilities, the projection methods \cite{elguedj2008and} or subdivision-stabilized NURBS discretization can be incorporated \cite{dortdivanlioglu2018mixed}.

The solution of high-order partial differential equations (PDEs) attracted a lot of attention in recent years. An important subclass is the $2n$-order PDE, which for $n=1,2,3$ reduces to different classical PDEs including Laplacian, Allen-Cahn, biharmonic, Cahn-Hilliard, Swift-Hohenberg, phase-field crystal and other problems. Previous work of numerical analysis typically focused on each of these problems separately (e.g., \cite{bleher1994existence, korzec2008stationary, verhoosel2011isogeometric, vignal2015energy}). Only a limited number of investigations exists for a general $2n$-order problem (e.g., for phase-field models \cite{vignal2017energy}). For example, non-degeneracy and uniqueness of its periodic solutions were studied in \cite{torres2013non}. 

Dispersion-minimizing methods based on modified integration rules for reducing dispersion error have been developed previously for classical FEM \cite{guddati2004modified, ainsworth2010optimally} and isogeometric analysis \cite{puzyrev2017dispersion, calo2017dispersion}. The dispersion error is reduced by blending two standard quadrature rules or using special quadrature rules \cite{deng2018dispersion}. For the standard finite and isogeometric elements, these optimal dispersion methods lead to two additional orders of error convergence (superconvergence) in the eigenvalues, while the eigenfunction errors do not degenerate.

In this paper, we utilize the mixed FEM framework for a general $2n$-order linear differential eigenvalue problem. We develop the mixed isogeometric framework for these eigenvalue problems and present error analysis for both eigenvalue and eigenfunctions. Optimal blending rules for the mixed isogeometric discretizations of the $2n$-order problem are presented up to $n=3$. 

The rest of this paper is organized as follows. Section 2 describes the differential eigenvalue problems under consideration. Section 3 presents the mixed isogeometric formulation. In Section 4, we present the optimally-blended rules and their eigenvalue error analysis. Section 5 shows numerical examples to demonstrate the performance of the method. Concluding remarks are given in Section 6.

\section{Problem statement} \label{sec:ps}
Let $\Omega \subset \mathbb{R}^d, d=1,2,3,$ be a bounded domain with Lipschitz boundary $\partial \Omega$. We consider a class of $2n$-order linear differential eigenvalue problems: Find $\lambda$ and non-zero $u$ satisfying 
\begin{equation} \label{eq:pde}
\mathcal{L}  u = \lambda u \quad \text{in} \quad \Omega
\end{equation}
and $u$ is subject to appropriate boundary conditions with the differential operator defined as
\begin{equation} \label{eq:operator}
\mathcal{L} = \sum_{m=0}^n a_m (-\Delta)^m,
\end{equation}
where $\Delta$ is the Laplacian and $a_m \in L^\infty(\Omega), m=0,1,\cdots, n$ with $n$ being a positive integer. For simplicity, we assume that $\Omega = [0,1]^d \in \mathbb{R}^d, d=1,2,3$ and $a_m$ are constants in the following discussions.

The operator $\mathcal{L}$ in the general equation \eqref{eq:pde} covers many high-order differential operators arising in sciences and engineering. In particular:
\begin{itemize}
\item For $n=1, a_0 = 0, a_1 = 1$, $\mathcal{L} = -\Delta$ and \eqref{eq:pde} reduces to the Laplacian (or linearized Allen-Cahn \cite{allen1979microscopic}) eigenvalue problem.
\item For $n=2, a_0 = a_1 = 0, a_2 = 1,$ $\mathcal{L} = \Delta^2$ and \eqref{eq:pde} becomes the biharmonic eigenvalue problem. 
\item For $n=2, a_0 = 0, a_1 = 1, a_2 = 1,$ $\mathcal{L} = \Delta^2 -\Delta$ and \eqref{eq:pde} is the linearized Cahn-Hilliard eigenvalue problem of fourth order \cite{cahn1958free}. 
\item For $n=2, a_0 = 1, a_1 = -2, a_2 = 1,$ $\mathcal{L} = (1+\Delta)^2$ and \eqref{eq:pde} is the linearized Swift-Hohenberg eigenvalue problem \cite{swift1977hydrodynamic}. 
\item For $n=3, a_0 = a_1 = 0, a_2=1, a_3 = 1,$ $\mathcal{L} = -\Delta (\Delta^2 -\Delta)$ and \eqref{eq:pde} is the linearized Cahn-Hilliard eigenvalue problem of sixth order \cite{savina2003faceting}.
\item For $n=3, a_0 = 0, a_1 = 1,a_2=-2, a_3 = 1,$ $\mathcal{L} = -\Delta(1+\Delta)^2$ and \eqref{eq:pde} is the linearized phase-field crystal eigenvalue problem \cite{elder2002modeling}. 
\end{itemize}

We focus on the interfacial energy operator for the phase-field models listed. We chose this eigenvalue approximation as the linearized bulk energy contribution can vary orders of magnitude and of sign, making a general analysis out of the scope of the present work. The differential eigenvalue problem \eqref{eq:pde} with constant coefficients $a_m, m=0, 1, \cdots, n,$ has the following eigenpairs $(\lambda_j, u_j), j=1,2,\cdots, $
\begin{equation} \label{eq:1dexactsol}
\begin{aligned}
\lambda_j & = a_n (j^2\pi^2)^n + a_{n-1} (j^2\pi^2)^{n-1} + \cdots + a_1 (j^2\pi^2)^1 + a_0,   \\
u_j & = C_1 \sin(j\pi x) + C_2 \cos(j\pi x)
\end{aligned}
\end{equation}
for 1D,
\begin{equation} \label{eq:2dexactsol}
\begin{aligned}
\lambda_{jl} & = a_n \big((j^2+l^2) \pi^2 \big)^n + \cdots + a_1 \big((j^2+l^2) \pi^2 \big)^1 + a_0,   \\
u_{jl} & = C_1 \sin(j\pi x)\sin(l\pi y) + C_2 \cos(j\pi x)\cos(l\pi y)
\end{aligned}
\end{equation}
for 2D, and
\begin{equation} \label{eq:3dexactsol}
\begin{aligned}
\lambda_{jlq} & = a_n \big((j^2+l^2+q^2) \pi^2 \big)^n + \cdots + a_1 \big((j^2+l^2+q^2) \pi^2 \big)^1 + a_0,   \\
u_{jlq} & = C_1 \sin(j\pi x)\sin(l\pi y)\sin(q\pi z) + C_2 \cos(j\pi x)\cos(l\pi y)\cos(q\pi z)
\end{aligned}
\end{equation}
for 3D with the constants $C_1, C_2 \in \mathbb{C}$ to be determined by the boundary conditions. If we impose homogeneous Dirichlet or Neumann boundary conditions, then we obtain that $C_2=0$ or $C_1=0$, respectively. Once one constant is determined, the other constant can be determined by normalization, that is, to normalize $u$ such that $(u, u)_{L^2(\Omega)} = 1$ where $(\cdot, \cdot)_{L^2(\Omega)} $ denotes the inner-product in $L^2(\Omega)$. In this paper, we consider homogeneous Dirichlet boundary conditions for the biharmonic eigenvalue problem (the case of a simply supported plate) and periodic boundary conditions for other eigenvalue problems.

\section{Mixed formulations}

To motivate the presentation of the mixed formulation for any $n$, we start with $n=2, a_0=a_1=0, a_2=1$ which \eqref{eq:pde} becomes the biharmonic eigenvalue problem. In this case, the differential equation \eqref{eq:pde} can be recast in a mixed form as a system of equation of second-order 
\begin{equation}\label{eq:pdee}
- \Delta u = \mu \qquad \text{and} \qquad - \Delta \mu = \lambda u
\end{equation}
and both $u$ and $\mu$ are subject to appropriate boundary conditions. This system of equations is referred to as problem with two unknown fields; see for example \cite{brezzi1974existence,ciarlet2002finite,falk1980error}. This new auxiliary parameter has physical meanings. For example, in structural mechanics, the new unknown $\mu$ represents the bending moment \cite{timoshenko1959theory}, while in fluid dynamics, when the Stokes equations for viscous steady flow is transformed using stream function this represents the vorticity \cite{andreev2005postprocessing}. 

If the domain $\Omega$ has smooth boundary or it is convex polygonal domain, then the eigenvalue problem \eqref{eq:pdee} has infinitely many solutions $(\lambda_j, (\mu_j, u_j))$ such that \cite{babuvska1991eigenvalue,davies1997lp,ishihara1978mixed}
\begin{equation}
\begin{aligned}
0  < \lambda_1 & \le \lambda_2 \le \lambda_3 \le \cdots < \infty, \\
\mu_j & = -\Delta u_j, \\
(u_j, u_k)_{L^2(\Omega)} & = \delta_{jk}, \qquad \forall j,k\ge1,
\end{aligned}
\end{equation}
with the Kronecker delta defined as $\delta_{jk} =1$ when $j=k$ and zero otherwise. Here, the eigenfunctions $u_j$ are orthonormal in $L^2(\Omega)$.

If $(\lambda_j, (\mu_j, u_j))$ is an eigenpair of \eqref{eq:pdee}, then $(\lambda_j, u_j)$ is an eigenpair of \eqref{eq:pde} and $\mu = -\Delta u.$ Hence the regularity of $(\mu_j, u_j)$ can be inferred from the regularity properties of problem \eqref{eq:pde}; see for example \cite{andreev2005postprocessing}.

\subsection{Mixed formulation at continuous level}
For arbitrary positive integer $n$, we set
\begin{equation} \label{eq:sub}
\psi^m = - \Delta \psi^{m-1},
\end{equation}
for $m = 1,2,\cdots, n-1$ with $\psi^0 =  u$.
These auxiliary unknowns allow us to recast the differential equation \eqref{eq:pde} into the mixed form as a system of equations of second-order
\begin{equation} \label{eq:mixedpde}
\begin{aligned}
- \Delta \psi^{m-1} - \psi^m & = 0, \ m=1, 2, \cdots, n-1, \\
- a_n \Delta \psi^{n-1} + \sum_{m=0}^{n-1} a_m \psi^m & = \lambda u.
\end{aligned}
\end{equation}

Similarly, we expect that if $(\lambda_j, (\psi^1_j, \psi^2_j, \cdots, \psi^{n-1}_j, u_j))$ is an eigenpair of \eqref{eq:mixedpde}, then $(\lambda_j, u_j)$ is an eigenpair of \eqref{eq:pde} with $\psi^m_j = -\Delta \psi^{m-1}_j$ as in \eqref{eq:sub}. The regularity of $(\psi^1_j, \psi^2_j, \cdots, \psi^{n-1}_j, u_j)$ can be inferred from the regularity properties of problem \eqref{eq:pde} and we assume sufficient regularity of the problem \eqref{eq:pde}.

Now, we present the mixed variational formulation for \eqref{eq:pde} at the continuous level. We denote the standard $L^2(\Omega)$-norm as $\| \cdot \|_{0,\Omega} \equiv \| \cdot \| \equiv \| \cdot \|_{L^2(\Omega)}$. We adopt the standard Sobolev spaces of integer index $s$, $H^s(\Omega)$ and $H^s_0(\Omega)$, equipped with the norm $\| \cdot \|_{s,\Omega} \equiv \| \cdot \|_{H^s(\Omega)}$; see \cite{adams1975sobolev,ciarlet2002finite}.

We define the bilinear forms
\begin{equation} \label{eq:bilinearforms}
a(v, w) = \int_\Omega  \nabla v \cdot \nabla w \ \text{d} x,  \qquad  b(v, w) = \int_\Omega  v  w \ \text{d} x,
\end{equation}
where $\nabla$ is the gradient operator. The bilinear form $a(\cdot, \cdot)$ and $b(\cdot, \cdot)$ is usually referred as stiffness and mass, respectively. 
These bilinear forms can be written alternatively as
\begin{equation}
a(v, w) = (\nabla v, \nabla w)_{L^2(\Omega)},  \qquad  b(v, w) = ( v, w)_{L^2(\Omega)}.
\end{equation}

At the continuous level, for simplicity, we assume that the differential equation \eqref{eq:pde} is subject to simply supported boundary conditions
\begin{equation} \label{eq:ssbdry}
u = \Delta u = \Delta^2 u = \cdots = \Delta^{n-1} u = 0.
\end{equation}
The mixed formulations for \eqref{eq:pde} or \eqref{eq:mixedpde} is: Find the eigenpairs $(\lambda, (\psi^1, \psi^2, \cdots, \psi^{n-1}, u))$ with $\lambda \in \mathbb{R}$ and $\psi^m, u \in H_0^1(\Omega), m=1, 2, \cdots, n-1,$ satisfying
\begin{equation} \label{eq:mf}
\begin{aligned}
a(\psi^{m-1}, w^m) - b(\psi^m, w^m)  & = 0, \ m=1, 2, \cdots, n-1,  \quad \forall w^m \in H_0^1(\Omega), \\
a_n \ a(\psi^{n-1}, v) + \sum_{m=0}^{n-1} a_m \ b(\psi^m, v) & = \lambda b(u, v),  \quad \forall v \in H_0^1(\Omega). \\
\end{aligned}
\end{equation}

\begin{remark}
One can also consider other boundary conditions. For example, for $n=2, a_0=a_1=0, a_2=1$, with homogeneous Dirichlet boundary conditions
\begin{equation} \label{eq:bdrycp}
u = \frac{\partial u}{\partial \bfs{n}} = 0 \quad \text{on} \quad \partial\Omega,
\end{equation}
the mixed variational formulation of the corresponding \eqref{eq:pde} is: Find the eigenpairs $(\lambda, (\psi^1, u))$ with $\lambda \in \mathbb{R}, \psi^1 \in H^1(\Omega)$, and $u \in H_0^1(\Omega)$ satisfying
\begin{equation} \label{eq:n2mf}
\begin{aligned}
a(u, w) - b(\psi^1, w)  & = 0,  \quad \forall w \in H^1(\Omega), \\
a(\psi^1, v) & = \lambda b(u, v),  \quad \forall v \in H_0^1(\Omega). \\
\end{aligned}
\end{equation}
Herein, $\bfs{n}$ denotes the outward unit normal to the boundary $\partial \Omega$. We refer the readers to \cite{andreev2005postprocessing} for details.

\end{remark}

\subsection{Mixed formulation at discrete level}
At discrete level, we specify a finite dimensional approximation space $V_h^p \subset H^1_0(\Omega)$ where $V_h^p = \text{span} \{\phi_a^p\}$ is the span of the B-spline or Lagrange (for FEM) basis functions $\phi_a^p$ of order $p$. Consequently, the mixed isogeometric analysis (or FEM) of \eqref{eq:pde} with simply supported boundary conditions \eqref{eq:ssbdry} is to seek $\lambda^h \in \mathbb{R}$ and $\psi^m_h, u_h \in V_h^p, m=1, 2, \cdots, n-1,$ satisfying
\begin{equation} \label{eq:vfh}
\begin{aligned}
a(\psi^{m-1}_h, w^m_h) - b(\psi^m_h, w^m_h)  & = 0, \ m=1, 2, \cdots, n-1,  \quad \forall w^m_h \in V_h^p, \\
a_n \ a(\psi^{n-1}_h, v_h) + \sum_{m=0}^{n-1} a_m \ b(\psi^m_h, v_h) & = \lambda^h b(u_h, v_h),  \quad \forall v_h \in V_h^p. \\
\end{aligned}
\end{equation}

The definition of the B-spline basis functions in one dimension is as follows. 
Let $X = \{x_0, x_1, \cdots, x_m \}$ be a knot vector with knots $x_j$, that is, a nondecreasing sequence of real numbers which are called knots.  The $j$-th B-spline basis function of degree $p$, denoted as $\theta^j_p(x)$, is defined as \cite{de1978practical,piegl2012nurbs}
\begin{equation} \label{eq:B-spline}
\begin{aligned}
\theta^j_0(x) & = 
\begin{cases}
1, \quad \text{if} \ x_j \le x < x_{j+1} \\
0, \quad \text{otherwise} \\
\end{cases} \\ 
\theta^j_p(x) & = \frac{x - x_j}{x_{j+p} - x_j} \theta^j_{p-1}(x) + \frac{x_{j+p+1} - x}{x_{j+p+1} - x_{j+1}} \theta^{j+1}_{p-1}(x).
\end{aligned}
\end{equation}

In this paper, for isogeometric analysis, we utilize the B-splines on uniform tensor-product meshes with non-repeating knots, that is, the B-splines with maximum continuity on uniform meshes, while for finite element method, we utilize the standard Lagrange basis functions. 
We approximate the eigenfunctions as a linear combination of the B-spline (or Lagrange) basis functions and substitute all the basis functions for $V_h^p$ in \eqref{eq:vfh}. This leads to a matrix eigenvalue problem, which is then solved numerically. We give more details of the structures of matrix eigenvalue problem in the following.

\subsection{Quadrature rules}
In practice, we evaluate the integrals involved in the bilinear forms $a(\cdot, \cdot) $ and $b(\cdot, \cdot)$ numerically, that is using quadrature rules. On a reference element $\hat K$, a quadrature rule is of the form
\begin{equation} \label{eq:qr}
\int_{\hat K} \hat f(\hat{\boldsymbol{x}}) \ \text{d} \hat{\boldsymbol{x}} \approx \sum_{l=1}^{N_q} \hat{\varpi}_l \hat f (\hat{n_l}),
\end{equation}
where $\hat{\varpi}_l$ are the weights, $\hat{n_l}$ are the nodes, and $N_q$ is the number of quadrature points. For each element $K$, we assume that there is an invertible map $\sigma$ such that $K = \sigma(\hat K)$, which leads to the correspondence between the functions on $K$ and $\hat K$. Assuming $J_K$ is the corresponding Jacobian of the mapping, \eqref{eq:qr} induces a quadrature rule over the element $K$ given by
\begin{equation} \label{eq:q}
\int_{K}  f(\boldsymbol{x}) \ \text{d} \boldsymbol{x} \approx \sum_{l=1}^{N_q} \varpi_{l,K} f (n_{l,K}),
\end{equation}
where $\varpi_{l,K} = \text{det}(J_K) \hat \varpi_l$ and $n_{l,K} = \sigma(\hat n_l)$. 
For simplicity, we denote by $G_l$ the $l-$point Gauss-Legendre quadrature rule,  by $L_l$ the $l-$point Gauss-Lobatto quadrature rule, 
by $R_l$ the $l-$point Gauss-Radau quadrature rule, 
and by $O_p$ the optimal blending scheme for the $p$-th order isogeometric analysis with maximum continuity. In one dimension, $G_l, L_l$, and $R_l$ fully integrate polynomials of order $2l-1, 2l-3,$ and $2l-2$, respectively \cite{bartovn2016optimal,bartovn2016gaussian,bartovn2017gauss}.  

Applying quadrature rules to \eqref{eq:vfh}, we obtain the approximate form
\begin{equation} \label{eq:vfho}
\begin{aligned}
a_h(\psi^{m-1}_h, w^m_h) - b_h(\psi^m_h, w^m_h)  & = 0, \ m=1, 2, \cdots, n-1,  \quad \forall w^m_h \in V_h^p, \\
a_n \ a_h(\psi^{n-1}_h, v_h) + \sum_{m=0}^{n-1} a_m \ b_h(\psi^m_h, v_h) & = \tilde \lambda^h b_h(u_h, v_h),  \quad \forall v_h \in V_h^p, \\
\end{aligned}
\end{equation}
where for $w,v \in V_h^p$
\begin{equation} \label{eq:ba}
 a_h(w, v) = \sum_{K \in \mathcal{T}_h} \sum_{l=1}^{N_q^1} \varpi_{l,K}^{(1)} \nabla w (n_{l,K}^{(1)} ) \cdot \nabla v (n_{l,K}^{(1)} )
\end{equation}
and
\begin{equation} \label{eq:bb}
 b_h(w, v) = \sum_{K \in \mathcal{T}_h} \sum_{l=1}^{N_q^2} \varpi_{l,K}^{(2)} w (n_{l,K}^{(2)} ) v (n_{l,K}^{(2)} )
\end{equation}
with $\{\varpi_{l,K}^{(1)}, n_{l,K}^{(1)} \}$ and $\{\varpi_{l,K}^{(2)}, n_{l,K}^{(2)} \}$ specifying two (possibly different) quadrature rules. 
Using quadrature rules, we can write the matrix eigenvalue problem as
\begin{equation} \label{eq:mevp}
  \begin{bmatrix}
  \bfs{K} & -\bfs{M} &   \bfs{0} &   \bfs{0} &   \bfs{0} \\
  \bfs{0} &  \bfs{K} & -\bfs{M} &   \bfs{0} &   \bfs{0} \\
  \vdots & \vdots &  \vdots & \ddots &  \vdots \\
 \bfs{0} &   \bfs{0} &    \cdots&  \bfs{K} & -\bfs{M} \\
 a_0 \bfs{M} &  a_1 \bfs{M} &  \cdots & a_{n-2} \bfs{M} &  a_{n-1} \bfs{M} + a_n \bfs{K}
  \end{bmatrix}
  \begin{bmatrix}
  \bfs{U} \\ 
  \bfs{\Psi^1} \\
    \vdots \\
   \bfs{\Psi^{n-2}} \\
   \bfs{\Psi^{n-1}} \\
  \end{bmatrix}
 = \tilde \lambda^h 
  \begin{bmatrix}
  \bfs{0} & \bfs{0} \\ 
  \bfs{M} & \bfs{0}
  \end{bmatrix}
    \begin{bmatrix}
  \bfs{U} \\ 
  \bfs{\Psi^1} \\
    \vdots \\
   \bfs{\Psi^{n-2}} \\
   \bfs{\Psi^{n-1}} \\
  \end{bmatrix},
\end{equation}
where $\bfs{K}_{ab} =  a_h(\phi_a^p, \phi_b^p), \bfs{M}_{ab} = b_h(\phi_a^p, \phi_b^p),$ and $\bfs{U}, \bfs{\Psi^j}, j=1,\cdots, n-1$ are the corresponding representation of the eigenvector as the coefficients of the basis functions. Similar to the standard second-order eigenvalue problem, $\bfs{K}$ and $\bfs{M}$ are referred to as the stiffness and mass matrices resulting from \eqref{eq:bilinearforms}, respectively. This matrix eigenvalue problem \eqref{eq:mevp} has a similar structure with the one obtained by hybrid high-order discretization; see \cite[Eqn. 3.13]{calo2017spectral}.

\subsection{Optimally blended quadrature rules} \label{sec:optrules}
The optimally-blended rules are developed and analyzed for isogeometric analysis in \cite{puzyrev2017dispersion,calo2017dispersion,deng2018dispersion,bartovn2017generalization} for $p \le 7$ and generalized to arbitrary order $p$ in \cite{deng2017dispersion}. We denote the following blendings of Gauss-Legendre rule $G_p, G_{p+1}$ and Gauss-Lobatto rule $L_{p+1}$ as
\begin{equation} \label{eq:tau}
\tau_{gg} G_{p+1} + (1-\tau_{gg}) G_p,  \qquad \tau_{gl} G_{p+1} + (1-\tau_{gl}) L_{p+1}, 
\end{equation}
where $\tau_{gg}$ and $\tau_{gl}$ are blending parameters.
\begin{table}[ht]
\centering 
\begin{tabular}{| c || c | c | c | c | c | }
\hline
$p$ & 1 &  2  & 3 & 4  \\[0.1cm] \hline
$\tau_{gg}$  & 2 & 2 & $\frac{13}{3}$ & 22 \\[0.1cm] \hline
$\tau_{gl}$ & $\frac{1}{2}$ & $\frac{1}{3}$ & $-\frac{3}{2}$ & $-\frac{79}{5}$ \\[0.1cm] \hline
\end{tabular}
\caption{Optimal blending parameters.} 
\label{tab:tau} 
\end{table}

Table \ref{tab:tau} shows the optimal blending parameters for $p\le 4$; see also \cite{calo2017dispersion,deng2017dispersion}. For optimal blending parameters of higher order $p$ and blending among other quadrature rules, we refer to \cite{deng2017dispersion}. These optimally-blended quadrature rules improve spectrum errors significantly, which we show in the next section.

\section{Eigenvalue errors}
In this section, following the eigenvalue error estimates established in \cite{deng2017dispersion,puzyrev2017dispersion,calo2017dispersion}  for the second-order Laplacian eigenvalue problem, we derive a priori eigenvalue error estimates for \eqref{eq:mevp}, which  can be viewed as a generalization to the $2n$-order differential eigenvalue problems. The mixed formulation helps us deliver the optimal error convergence rates for the biharmonic eigenvalue problem.


Following the structure in \cite{deng2017dispersion}, 
we seek an approximation of eigenfunction $u_h$ in the form 
\begin{equation} \label{eq:lc}
\sum_{j = 0}^N U^j \theta^j_p(x),
\end{equation}
where $U^j$ are the unknown coefficients which corresponds to the the $p$-th order polynomial approximation which are to be determined, that is the component of the unknown vector $\bfs{U}$ in \eqref{eq:mevp}.
Using the Bloch wave assumption \cite{odeh1964partial}, we write 
\begin{equation} \label{eq:bloch}
U^j = e^{ij \omega h },
\end{equation}
 where $i^2 = -1$ and $\omega$ is an approximated frequency. In the view of auxiliary fields \eqref{eq:sub}, we denote $\psi^m_h = u_h$ and $\bfs{\Psi}^m = \bfs{U}$ when $m=0.$ In the Bloch wave assumption \eqref{eq:bloch}, $jh$ resembles the spatial variable $x$. Using the auxiliary fields defined \eqref{eq:sub}, this allows us further assume Bloch wave solutions for their derivatives, that is
\begin{equation} \label{eq:blochdiff}
\Psi^{m,j} = \omega^{2m} e^{ij \omega h }, m=0,1, 2, \cdots, n-1,
\end{equation}
where $\Psi^{m,j}$ is the component of the unknown vector $\bfs{\Psi}^m$ in \eqref{eq:mevp}. We observe that \eqref{eq:blochdiff} recovers \eqref{eq:bloch} when $m=0$.

The $C^{p-1}$ B-spline basis function $\theta^j_p$ has a support over $p+1$ elements. Thus, for $m=0,1, 2, \cdots, n-1,$ we have
\begin{equation}
\begin{aligned}
a_h(\psi^m_h, \theta^j_p) & = a_h \Big(\sum_{k, |k-j| \le p} \Psi^{m,j}_p \theta^k_p, \theta^j_p \Big) = A_p\Psi^m_p/h, \\
b_h(\psi^m_h, \theta^j_p) & = b_h \Big(\sum_{k, |k-j| \le p} \Psi^{m,j}_p \theta^k_p, \theta^j_p \Big) = B_p\Psi^m_p \ h, \\
\end{aligned}
\end{equation}
where $a_h(\cdot, \cdot), b_h(\cdot, \cdot)$ approximate (or exactly-integrate when we use appropriate quadrature rules) bilinear forms and
\begin{equation}
\begin{aligned}
\Psi^m_p & = [\Psi^{m,j-p}_p \quad \Psi^{m,j-p+1}_p \quad \cdots \quad \Psi^{m,j}_p \quad \cdots \quad \Psi^{m,j+p-1}_p \quad \Psi^{m,j+p}_p ]^T, \\
A_p & = [A^{j-p}_p \quad A^{j-p+1}_p \quad \cdots \quad A^{j}_p \quad \cdots \quad A^{j+p-1}_p \quad A^{j+p}_p ], \\
B_p & = [B^{j-p}_p \quad B^{j-p+1}_p \quad \cdots \quad B^{j}_p \quad \cdots \quad B^{j+p-1}_p \quad B^{j+p}_p ], \\
\end{aligned}
\end{equation}
with 
\begin{equation} \label{eq:ahbh}
A^{j-k}_p = a_h(\theta^{j-k}_p, \theta^j_p)h, \qquad B^{j-k}_p = b_h(\theta^{j-k}_p, \theta^j_p)/h
\end{equation}
 for $k=p, p-1, \cdots, -p$. For $m=0,$ we also denote $\Psi^m_p = U_p$ with
 \begin{equation}
 U_p = [U^{j-p}_p \quad U^{j-p+1}_p \quad \cdots \quad U^{j}_p \quad \cdots \quad U^{j+p-1}_p \quad U^{j+p}_p ]^T.
 \end{equation}

 The symmetry of the B-spline basis functions (on uniform meshes and away from the boundaries) further implies that 
\begin{equation} \label{eq:symm}
A^{j-k}_p = A^{j+k}_p, \qquad B^{j-k}_p = B^{j+k}_p.
\end{equation}

Thus, using this symmetry, Euler's formula, Bloch wave assumptions \eqref{eq:bloch} and \eqref{eq:blochdiff}, we can deduce the following
\begin{equation} \label{eq:abeh}
\begin{aligned}
a_h(\psi^m_h, \theta^j_p) & = A_p\Psi^m_p/h = \omega^{2m} \big(A^j_p + 2\sum_{k=1}^p A^{j+k}_p \cos(k\omega h) \big) e^{ij \omega h } / h = \omega^{2m} A_pU_p/h, \\
b_h(\psi^m_h, \theta^j_p) & = B_p\Psi^m_p \ h = \omega^{2m}  \big(B^j_p + 2\sum_{k=1}^p B^{j+k}_p \cos(k\omega h) \big) e^{ij \omega h } h = \omega^{2m} B_pU_p \ h. \\
\end{aligned}
\end{equation}


\begin{lemma} \label{lem:abpower}
Denote $\Lambda = \omega h.$ For any positive integer $p$, 
denoting
\begin{equation}
C_{p+2} = 2 (-1)^p \Big( \sum_{k=1}^p \frac{k^{2p+4}}{(2p+4)!} A^{j+k}_p  + \frac{k^{2p+2}}{(2p+2)!} B^{j+k}_p  \Big),
\end{equation}
there holds
\begin{equation}
\begin{aligned}
\frac{A_p U_p }{B_p U_p} = \Lambda^2 + C_{p+2} \Lambda^{2p+4} + \mathcal{O}(\Lambda^{2p+6}),
\end{aligned}
\end{equation}
when $a_h(\cdot, \cdot)$ and $b_h(\cdot, \cdot)$ in \eqref{eq:ahbh} are approximated using the optimally blended quadrature rules.
\end{lemma}
\begin{proof}
We omit the proof as the identity is proved using the same arguments in the proof of Lemma 6 supplied with Lemma 7 in the paper \cite{deng2017dispersion}.
\end{proof}

\begin{theorem} \label{thm:main}
Assuming $a_h(\cdot, \cdot)$ and $b_h(\cdot, \cdot)$ in \eqref{eq:ahbh} are approximated using the optimally blended quadrature rules, there holds 
\begin{equation}
\begin{aligned}
\tilde \lambda^h - \sum_{k=0}^n a_k  \omega^{2k} =  \Big(C_{p+2} \sum_{k=1}^n a_k  \omega^{2k} \Big) (\omega h)^{2p+2} + \mathcal{O}(h)^{2p+4}.
\end{aligned}
\end{equation}
\end{theorem}

\begin{proof}
Let $\theta^j_p$ be a test function for each equation in \eqref{eq:vfho}. We represent the approximated eigenfunctions $\psi^m_h, m=0,1, \cdots, n-1,$ in the the same way as \eqref{eq:lc}. Substituting all these terms into \eqref{eq:vfho}, we obtain
\begin{equation}
\begin{aligned}
A_p\Psi^{m-1}_p/h  - B_p\Psi^m_p \ h  & = 0, \ m=1, 2, \cdots, n-1, \\
a_n A_p\Psi^{n-1}_p/h + \sum_{m=0}^{n-1} a_m \ B_p \Psi^m_p \ h & = \tilde \lambda^h B_p\Psi^0_p \ h. \\
\end{aligned}
\end{equation}
After substituting the first equation into the second one, simple manipulations yield
\begin{equation}
\sum_{m=1}^{n} a_m A_p\Psi^{m-1}_p  = (\tilde \lambda^h - a_0 ) B_p\Psi^0_p \ h^2,
\end{equation}
which,  by using \eqref{eq:abeh}, is further simplified as
\begin{equation}
\sum_{m=1}^{n} a_m \omega^{2m-2} A_p U_p  = (\tilde \lambda^h - a_0 ) B_p U_p \ h^2.
\end{equation}

Applying Lemma \ref{lem:abpower}, we have
\begin{equation}
\frac{(\tilde \lambda^h - a_0 ) h^2}{\sum_{m=1}^{n} a_m \omega^{2m-2}} = \frac{ A_p U_p}{B_p U_p } = (\omega h)^2 + C_{p+2} (\omega h)^{2p+4} + \mathcal{O}(\omega h)^{2p+6},
\end{equation}
which is further simplified as
\begin{equation}
\tilde \lambda^h - a_0 = \sum_{m=1}^{n} a_m \omega^{2m} + \Big( C_{p+2} \sum_{m=1}^{n} a_m \omega^{2m} \Big) (\omega h)^{2p+2} + \mathcal{O}(h^{2p+4}).
\end{equation}
Rewriting the equation completes the proof.
\end{proof}

\begin{remark}
In the case where the true eigenvalue can be rewritten in form of $\lambda = \sum_{k=0}^n a_k  \omega^{2k}$ as in \eqref{eq:1dexactsol}, \eqref{eq:2dexactsol}, and \eqref{eq:3dexactsol}, the mixed formulation with optimally-blended rules leads to an improved eigenvalue error of rate $| \lambda_{h} - \lambda | \approx \mathcal{O}(h^{2p+2})$. When standard quadrature rules are applied, the method retains its optimal rates $| \lambda_{h} - \lambda | \approx \mathcal{O}(h^{2p})$. This theorem establishes that the mixed formulation for $2n$-order differential eigenvalue maintains the theoretical findings established for the approximation method of the second-order Laplacian eigenvalue problem. For the generalization to multiple dimensions, we refer to \cite{deng2017dispersion} for the second-order Laplacian eigenvalue problem.
\end{remark}

\section{Numerical experiments} \label{sec:num}
In this section, we present 1D, 2D, and 3D numerical results which cover the spectral approximations of the biharmonic, Cahn-Hilliard, Swift-Hohenberg, and phase-field crystal operators. We limit our studies to simple geometrical domains and uniform meshes to focus our attention on the numerical aspects of the problem. We utilize both mixed isogeometric and finite elements with $p=1,2,3,4$. For our simulations, we consider the unitary domain $\Omega = [0,1]^d$. The exact solutions of \eqref{eq:pde} are given in \eqref{eq:1dexactsol}--\eqref{eq:3dexactsol} with the parameters specified in Section \ref{sec:ps}.

First of all, we present the matrix eigenvalue problems associated with the operators described above, which are easily obtained from the general representation \eqref{eq:mevp}. In particular, with slight abuse of notation, we have
\begin{equation} \label{eq:mevp_bh}
  \begin{bmatrix}
  \bfs{K} & -\bfs{M} \\
  \bfs{0} &  \bfs{K} \\
  \end{bmatrix}
  \begin{bmatrix}
  \bfs{U} \\ 
  \bfs{\Psi^1} \\
  \end{bmatrix}
 = \tilde \lambda^h 
  \begin{bmatrix}
  \bfs{0} & \bfs{0} \\ 
  \bfs{M} & \bfs{0}
  \end{bmatrix}
    \begin{bmatrix}
  \bfs{U} \\ 
  \bfs{\Psi^1} \\
  \end{bmatrix}
\end{equation}
for biharmonic eigenvalue problem, 
\begin{equation} \label{eq:mevp_ch}
  \begin{bmatrix}
  \bfs{K} & -\bfs{M} \\
  \bfs{0} &  \bfs{K} + \bfs{M}\\
  \end{bmatrix}
  \begin{bmatrix}
  \bfs{U} \\ 
  \bfs{\Psi^1} \\
  \end{bmatrix}
 = \tilde \lambda^h 
  \begin{bmatrix}
  \bfs{0} & \bfs{0} \\ 
  \bfs{M} & \bfs{0}
  \end{bmatrix}
    \begin{bmatrix}
  \bfs{U} \\ 
  \bfs{\Psi^1} \\
  \end{bmatrix}
\end{equation}
for Cahn-Hilliard eigenvalue problem, 
\begin{equation} \label{eq:mevp_sh}
  \begin{bmatrix}
  \bfs{K} & -\bfs{M} \\
  \bfs{M} &  \bfs{K} - 2\bfs{M}\\
  \end{bmatrix}
  \begin{bmatrix}
  \bfs{U} \\ 
  \bfs{\Psi^1} \\
  \end{bmatrix}
 = \tilde \lambda^h 
  \begin{bmatrix}
  \bfs{0} & \bfs{0} \\ 
  \bfs{M} & \bfs{0}
  \end{bmatrix}
    \begin{bmatrix}
  \bfs{U} \\ 
  \bfs{\Psi^1} \\
  \end{bmatrix}
\end{equation}
for Swift-Hohenberg eigenvalue problem, 
and
\begin{equation} \label{eq:mevp_pfc}
  \begin{bmatrix}
  \bfs{K} & -\bfs{M} &   \bfs{0}  \\
  \bfs{0} &  \bfs{K} & -\bfs{M}  \\
  \bfs{0} & \bfs{M} &  \bfs{K} - 2\bfs{M}
  \end{bmatrix}
  \begin{bmatrix}
  \bfs{U} \\ 
  \bfs{\Psi^1} \\
   \bfs{\Psi^2} \\
  \end{bmatrix}
 = \tilde \lambda^h 
  \begin{bmatrix}
  \bfs{0} & \bfs{0} & \bfs{0} \\ 
    \bfs{0} & \bfs{0} & \bfs{0} \\ 
  \bfs{M} & \bfs{0} & \bfs{0} \\
  \end{bmatrix}
  \begin{bmatrix}
  \bfs{U} \\ 
  \bfs{\Psi^1} \\
   \bfs{\Psi^2} \\
  \end{bmatrix}
  \end{equation}
for phase-field crystal  eigenvalue problem, respectively. For the fourth-order differential eigenvalue problem, one can symmetrize the matrix on the left-hand side, that is, the matrix eigenvalue problem \eqref{eq:mevp_bh} to \eqref{eq:mevp_sh} are equivalent to
\begin{equation} \label{eq:mevpbh}
  \begin{bmatrix}
  \bfs{K} & -\bfs{M} \\
  -\bfs{M} &  \bfs{K} \\
  \end{bmatrix}
  \begin{bmatrix}
  \bfs{U} \\ 
  \bfs{\Psi^1} \\
  \end{bmatrix}
 = (\tilde \lambda^h -1)
  \begin{bmatrix}
  \bfs{0} & \bfs{0} \\ 
  \bfs{M} & \bfs{0}
  \end{bmatrix}
    \begin{bmatrix}
  \bfs{U} \\ 
  \bfs{\Psi^1} \\
  \end{bmatrix},
\end{equation}
\begin{equation} \label{eq:mevpch}
  \begin{bmatrix}
  \bfs{K} & -\bfs{M} \\
  -\bfs{M} &  \bfs{K} + \bfs{M}\\
  \end{bmatrix}
  \begin{bmatrix}
  \bfs{U} \\ 
  \bfs{\Psi^1} \\
  \end{bmatrix}
 = (\tilde \lambda^h - 1)
  \begin{bmatrix}
  \bfs{0} & \bfs{0} \\ 
  \bfs{M} & \bfs{0}
  \end{bmatrix}
    \begin{bmatrix}
  \bfs{U} \\ 
  \bfs{\Psi^1} \\
  \end{bmatrix},
\end{equation}
and
\begin{equation} \label{eq:mevpsh}
  \begin{bmatrix}
  \bfs{K} & -\bfs{M} \\
  -\bfs{M} &  \bfs{K} - 2\bfs{M}\\
  \end{bmatrix}
  \begin{bmatrix}
  \bfs{U} \\ 
  \bfs{\Psi^1} \\
  \end{bmatrix}
 = (\tilde \lambda^h -2)
  \begin{bmatrix}
  \bfs{0} & \bfs{0} \\ 
  \bfs{M} & \bfs{0}
  \end{bmatrix}
    \begin{bmatrix}
  \bfs{U} \\ 
  \bfs{\Psi^1} \\
  \end{bmatrix},
\end{equation}
respectively. Numerical experiments show that it takes less time to solve \eqref{eq:mevpbh} to \eqref{eq:mevpsh} than solving the original problems \eqref{eq:mevp_bh} to \eqref{eq:mevp_sh} that have non-symmetric matrices on the left-hand side. We use the symmetric systems to solve the differential eigenvalue problems numerically in the following experiments.

As the eigenvalues of \eqref{eq:pde} can be large, we present the relative eigenvalue errors, which for a $j$-th eigenvalue, is defined as
\begin{equation}
e_{\lambda_j} = \frac{| \tilde \lambda^h_j - \lambda_j |}{\lambda_j}.
\end{equation}

We start our accuracy and convergence studies in 1D and then extend them to multiple dimensions. Figure \ref{fig:fig1} shows the eigenvalue errors for isogeometric elements of maximum continuity for the biharmonic, Cahn-Hilliard and Swift-Hohenberg equations with homogeneous Dirichlet boundary conditions. The exact solutions are in form of \eqref{eq:1dexactsol} in 1D. Similar plots have been previously shown in \cite{cottrell2006isogeometric} for approximate eigenvalues of second-order PDEs. As we can see in Figure \ref{fig:fig1}, all three fourth-order operators have quite similar errors when the same boundary conditions are used; the same happens also in multiple dimensions, as we show below.

The case of homogeneous Dirichlet boundary conditions has a particular importance for the biharmonic equation where it corresponds to the case of a simply supported plate. For the differential eigenvalue problems that describe the phase separation process, periodic boundary conditions are more relevant. In the following figures, we consider Dirichlet boundary conditions for the biharmonic equation and periodic boundary conditions for Cahn-Hilliard, Swift-Hohenberg, and phase-field crystal equations.

\begin{figure}[!ht]
\centering\includegraphics[width=13cm]{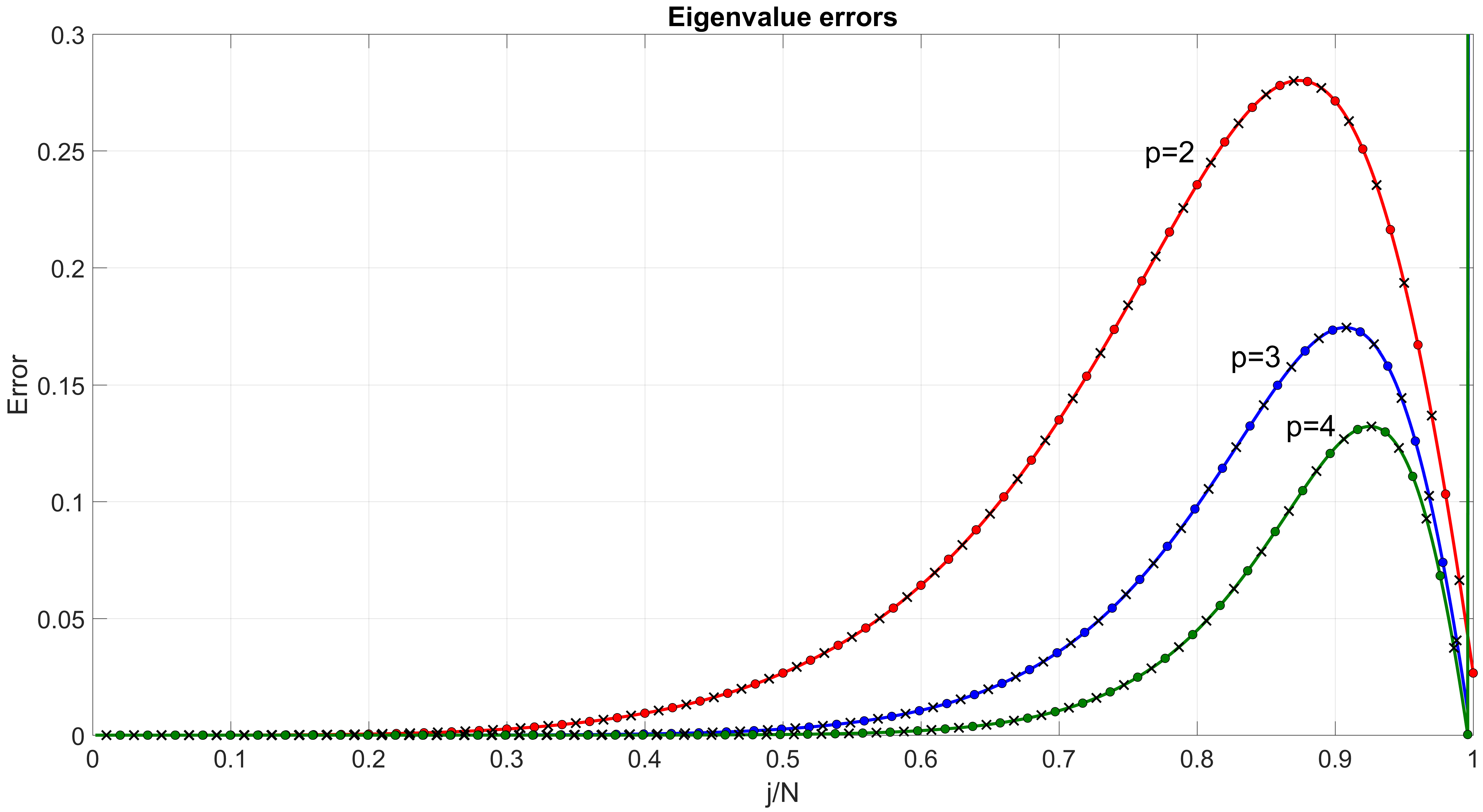}
\caption{Relative approximation errors for quadratic, cubic and quartic isogeometric elements for the biharmonic (solid lines), Cahn-Hilliard (circles) and Swift-Hohenberg (crosses) equations in 1D.}
\label{fig:fig1}
\end{figure}

Figure \ref{fig:fig2} shows the approximation errors for 1D Cahn-Hilliard and phase-field crystal operators with periodic boundary conditions for $p=2,3,4$. Swift-Hohenberg spectra are found to be very similar to the Cahn-Hilliard results and are omitted here for brevity. Once again, we can see a clear improvement in the spectral accuracy of IGA discretizations with an increase in $p$. Due to the use of periodic boundary conditions, outlier modes (large spikes in the errors for $j/N$ close to one (high-frequencies) in the spectra of high-order isogeometric discretizations) are absent in this case. This is related to the fact that outliers are caused by the basis functions with support on the boundaries of the domain \cite{puzyrev2017spectral}, which are absent in the periodic case.

\begin{figure}[!ht]
\centering
 \subfigure{\includegraphics[width=6.5cm]{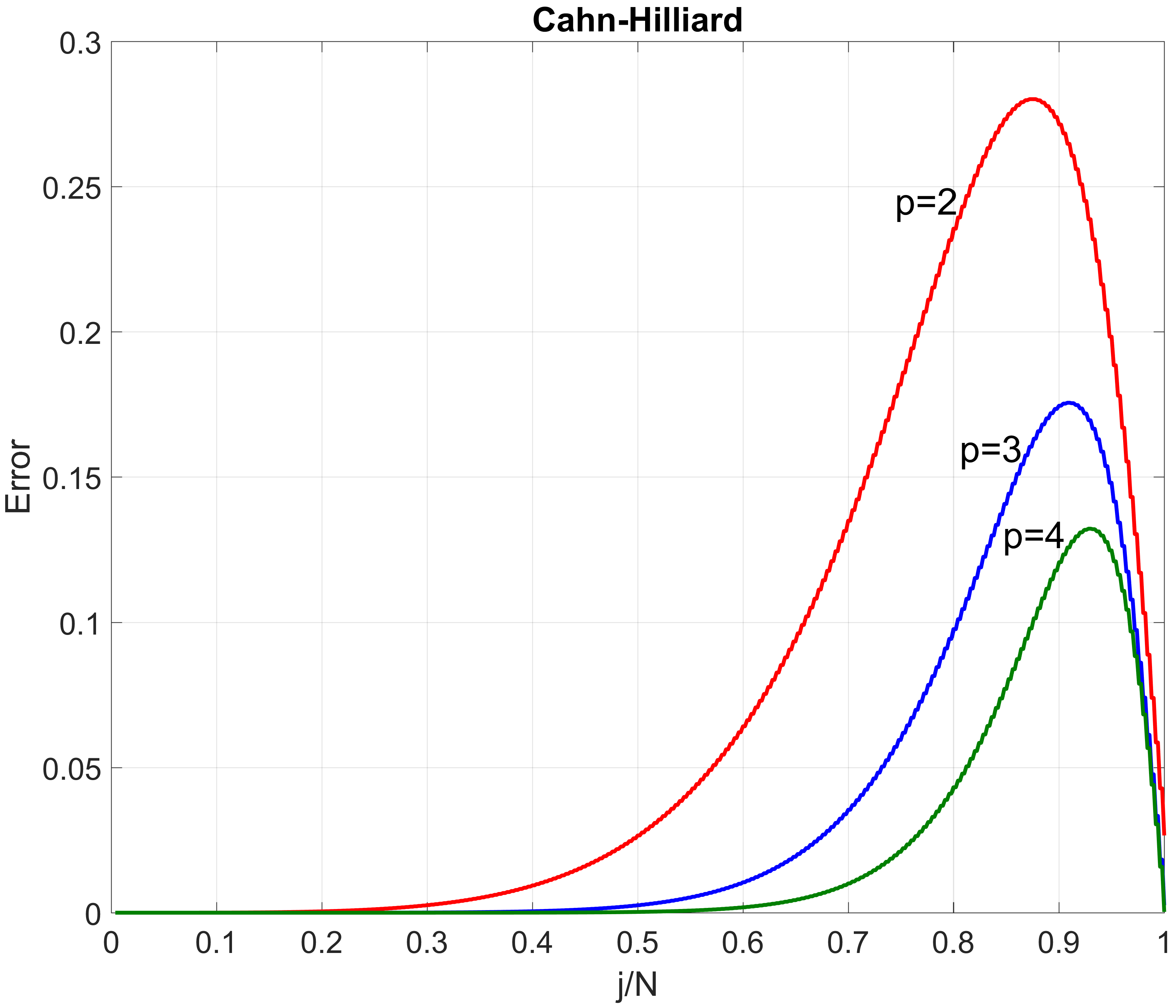}}
 \hspace{0.05cm}
 \subfigure{\includegraphics[width=6.5cm]{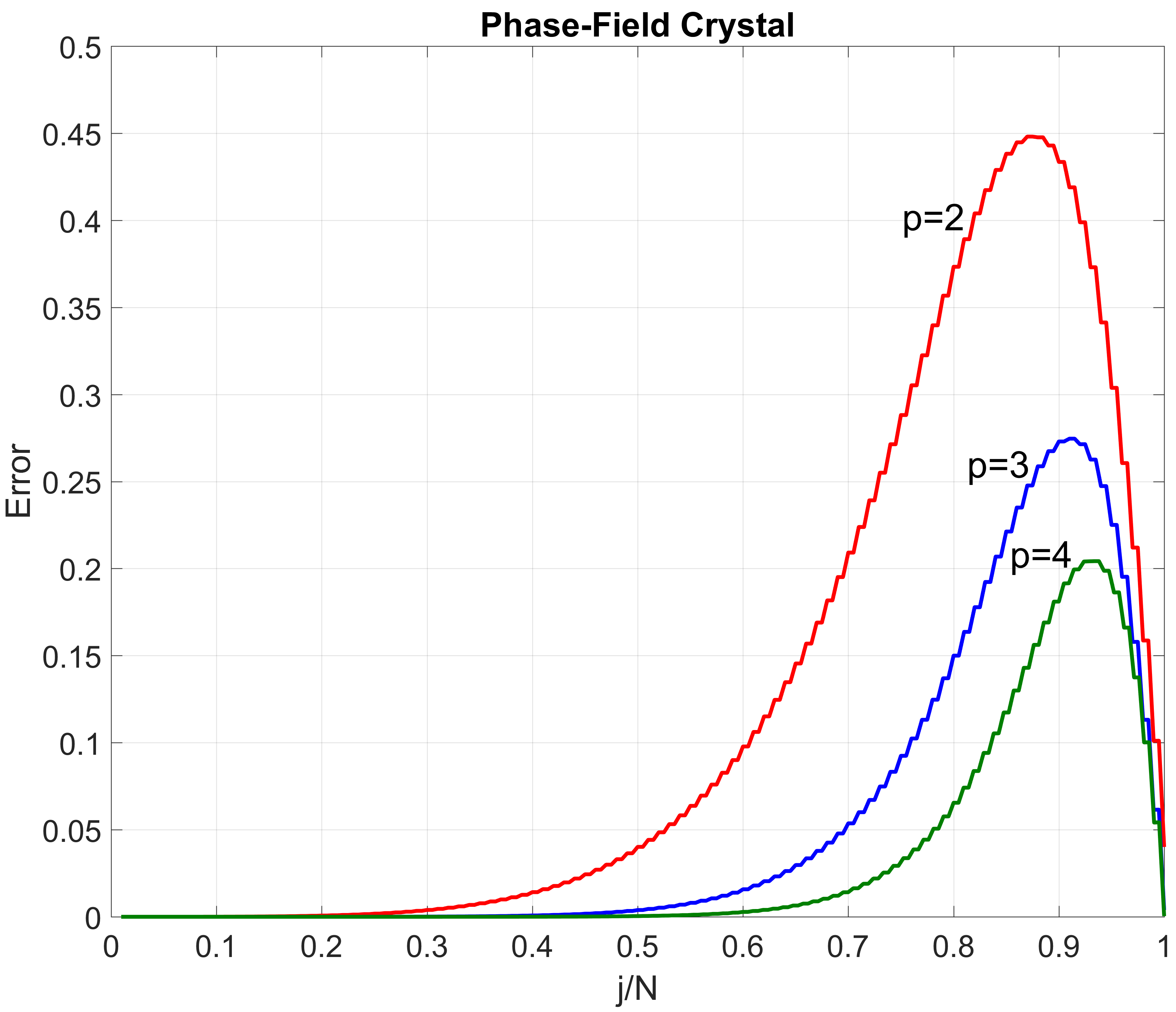}}
\caption{Relative approximation errors for Cahn-Hilliard and phase-field crystal operators with periodic boundary conditions for $p=2,3,4$. Note different Y scales across panels.}
\label{fig:fig2}
\end{figure}

Figure \ref{fig:fig3} compares the eigenvalue and eigenfunction errors between the mixed $C^0$ finite elements and $C^1$ isogeometric elements. We observe branching of the finite element spectrum which is typical for high-order $C^0$ discretizations. Notably, B-spline basis functions of the highest $p-1$ continuity do not exhibit such branching patterns on uniform meshes. Similar to the standard (non-mixed) formulation, the large spikes in the approximation errors in the middle of the spectra at the transition point between the acoustic and optical branches are absent in mixed isogeometric discretizations.

\begin{figure}[!ht]
\centering\includegraphics[width=13cm]{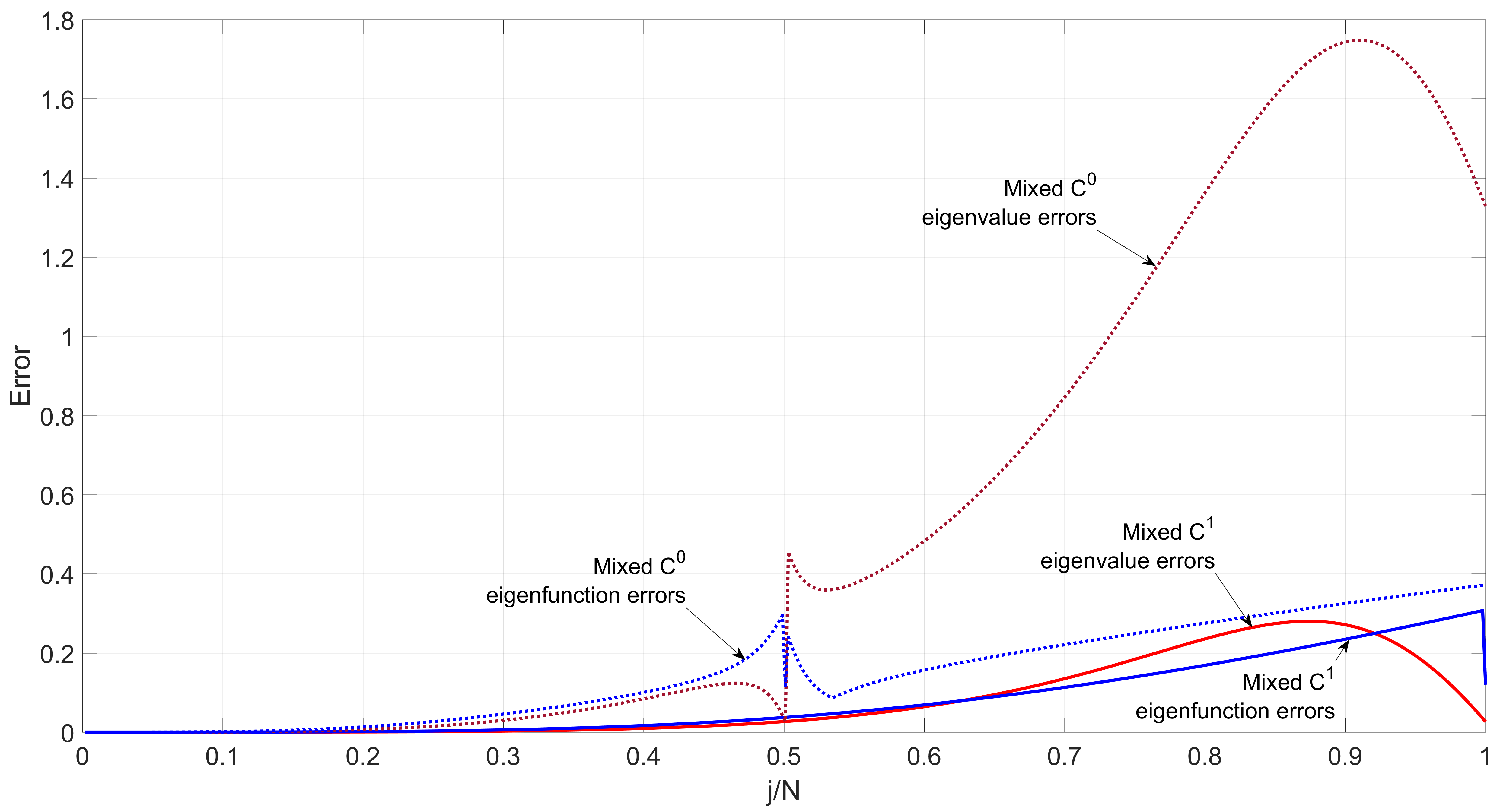}
\caption{Relative eigenvalue and eigenfunction errors for quadratic $C^0$ and $C^1$ elements for the 1D biharmonic equation.}
\label{fig:fig3}
\end{figure}

In Figure \ref{fig:fig5}, we compare the approximation errors for quadratic elements when using mixed isogeometric analysis with standard Gauss quadratures and optimally-blended rules \eqref{eq:tau}. In this case, we also use Dirichlet boundary conditions for the biharmonic equation and periodic boundary conditions for Cahn-Hilliard and phase-field crystal equations. Despite this, all three cases exhibit similar behaviour. The optimally-blended rules lead to the convergence rates of order $2p+2$ as predicted by the theory of subsection \ref{sec:optrules}.

\begin{figure}[!ht]
\centering
 \subfigure{\includegraphics[width=6.5cm]{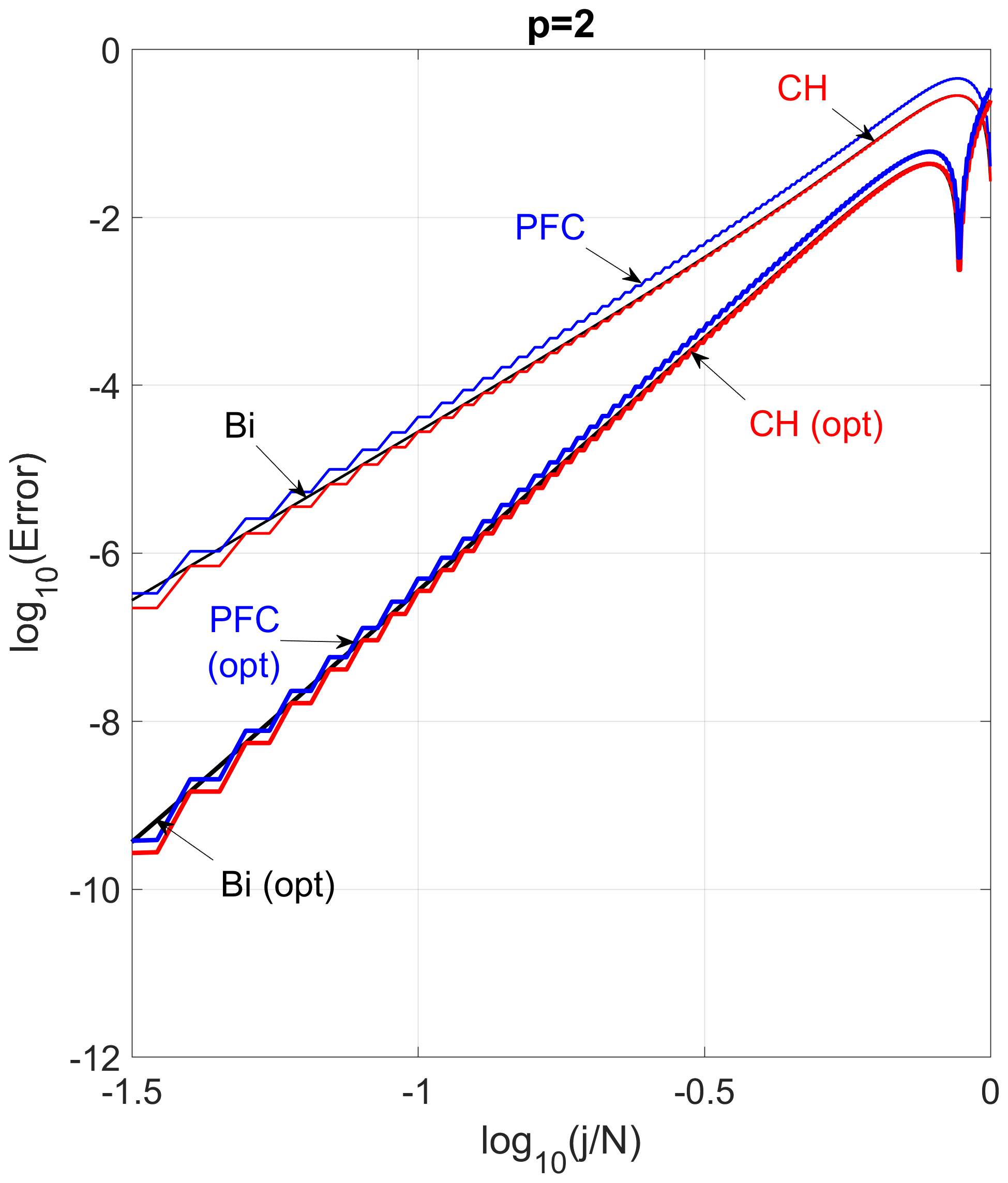}}
 \hspace{0.05cm}
 \subfigure{\includegraphics[width=6.5cm]{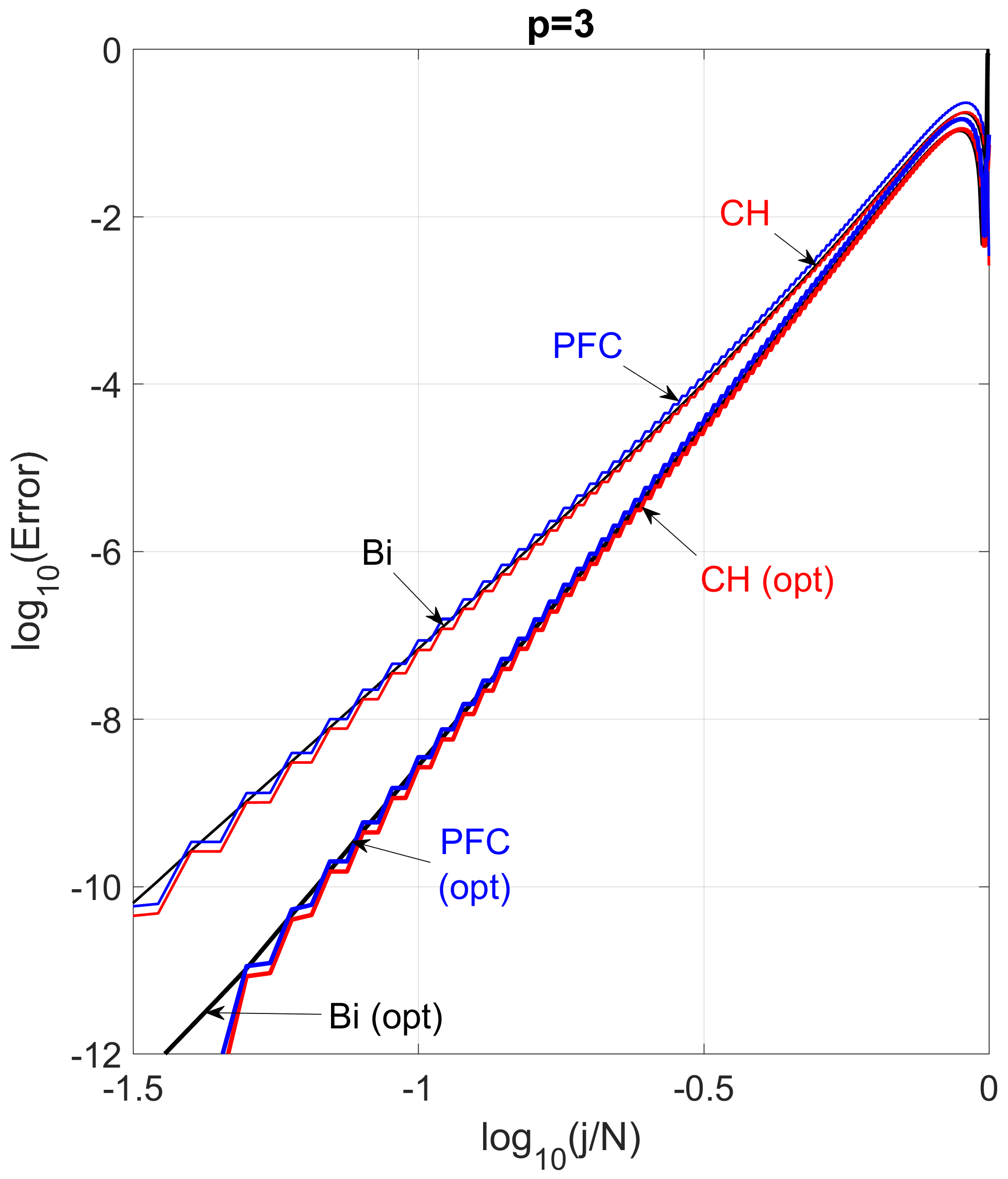}}
\caption{Relative approximation errors for the 1D biharmonic equation with Dirichlet boundary conditions and Cahn-Hilliard and phase-field crystal equations with periodic boundary conditions using standard Gauss and optimal quadrature rules.}
\label{fig:fig5}
\end{figure}

Figure \ref{fig:fig6} compares the relative approximation errors for the 3D biharmonic eigenvalue problem. In the multidimensional plots shown below, the axes correspond to eigenvalue indices $j, l, q$ in \eqref{eq:2dexactsol} and \eqref{eq:3dexactsol}. Again, using mixed isogeometric analysis with optimally-blended rules leads to smaller errors when compared to the fully-integrated case that employs standard Gauss quadratures.

\begin{figure}[!ht]
\centering
\includegraphics[width=13cm]{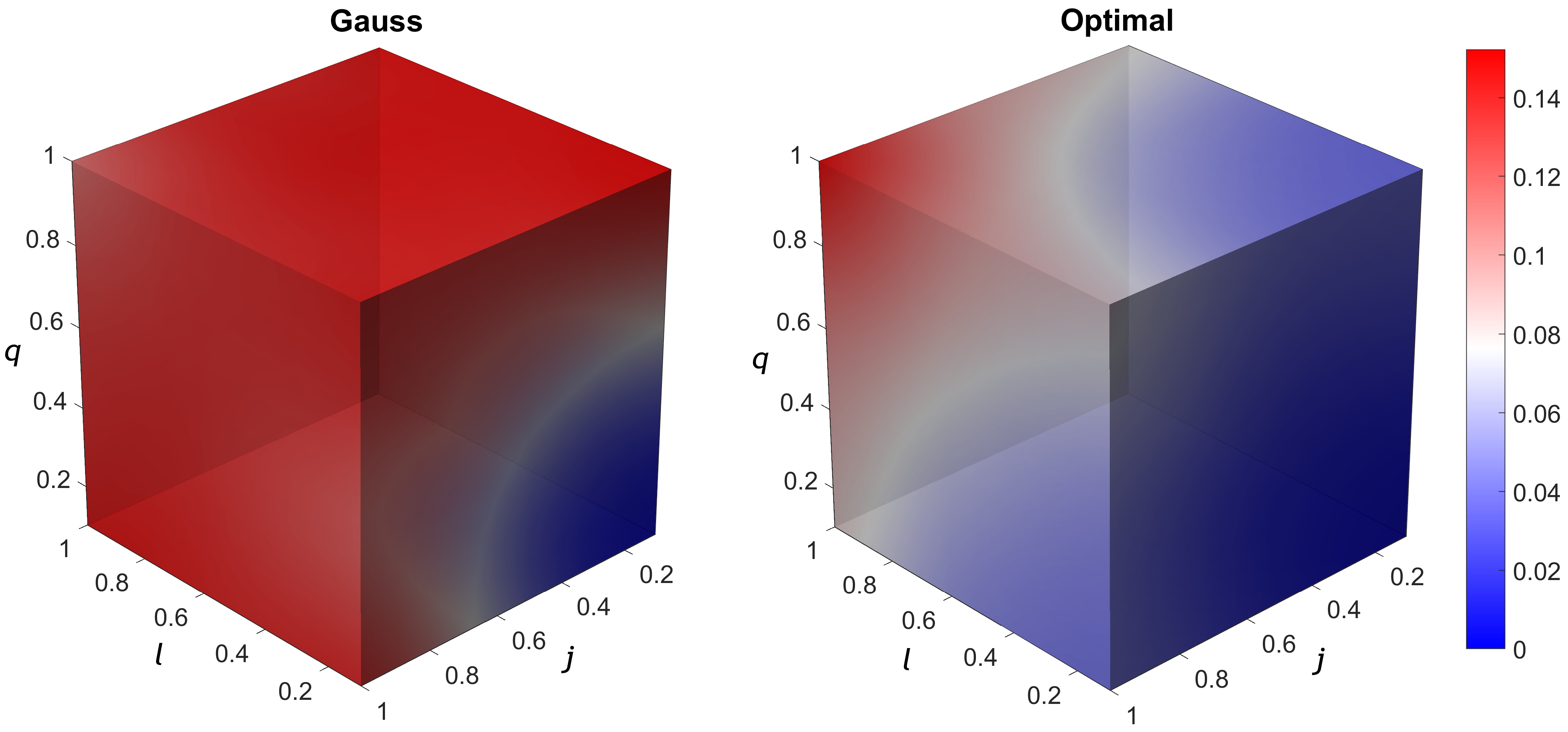}
\caption{Relative approximation errors for quadratic $C^1$ elements for the biharmonic equations in 3D using mixed isogeometric analysis with Gauss quadratures (left) and optimally-blended rules (right).}
\label{fig:fig6}
\end{figure}

Tables \ref{tab:bhev} to \ref{tab:pfcev} show the first, second, fourth, and eighth eigenvalue errors when using mixed isogeometric analysis with standard Gauss quadratures and optimally-blended rules for linear, quadratic, and cubic elements. In these tables, we denote by G when using the standard Gauss quadrature rule while by O when the optimally-blended rule is applied. There are different optimally-blended rules and they lead to the same numerical results; see \cite{puzyrev2017dispersion,calo2017dispersion}. Herein, we use the $G_{p+1}$-$L_{p+1}$ optimally-blended rules. We also denote the convergence rates as $\rho_p$ for $p$-th order elements. For $p$-th order elements, we obtain convergence rates of order $2p+2$ when using the optimally-blended rules. These numerical experiments verify our theoretical findings.

\begin{table}[h!]
\centering 
\begin{tabular}{| c | c | cc | cc | cc | cc |}
\hline
$p$ & $N$ & \multicolumn{2}{c|}{$e_{\lambda_1}$} & \multicolumn{2}{c|}{$e_{\lambda_2}$} & \multicolumn{2}{c|}{$e_{\lambda_4}$} & \multicolumn{2}{c|}{$e_{\lambda_8}$} \\
 & & G &  O & G &  O & G &  O & G &  O  \\[0.1cm] \hline
 & 4&	1.08e-1&	3.24e-3&	4.00e-1&	4.39e-2&	4.78e-1&	5.39e-2&	8.55e-1&	2.00e-1 \\
 &8&	2.60e-2&	1.99e-4&	9.10e-2&	2.63e-3&	1.08e-1&	3.24e-3&	2.08e-1&	1.26e-2 \\
1&16&	6.44e-3&	1.24e-5&	2.21e-2&	1.62e-4&	2.60e-2&	1.99e-4&	4.90e-2&	7.65e-4 \\
 &32&	1.61e-3&	7.74e-7&	5.48e-3&	1.01e-5&	6.44e-3&	1.24e-5&	1.21e-2&	4.74e-5 \\
& $\rho_1$	& 2.02&	4.01&	2.06&	4.03&	2.07&	4.03&	2.05&	4.02 \\[0.1cm] \hline

 & 4&	1.20e-3&	8.68e-5&	2.15e-2&	4.45e-3&	2.66e-2&	5.54e-3&	1.34e-1&	3.00e-2 \\
 &8&	6.83e-5&	1.34e-6&	9.74e-4&	6.97e-5&	1.20e-3&	8.68e-5&	5.23e-3&	7.18e-4 \\
2&16&	4.16e-6&	2.09e-8&	5.54e-5&	1.08e-6&	6.83e-5&	1.34e-6&	2.70e-4&	1.10e-5 \\
 &32&	2.59e-7&	3.26e-10&	3.38e-6&	1.68e-8&	4.16e-6&	2.09e-8&	1.60e-5&	1.71e-7 \\
 & $\rho_2$	& 4.06&	6.01&	4.20&	6.01&	4.21&	6.01&	4.34&	5.83 \\[0.1cm] \hline
 
 & 4&	1.94e-5&	3.44e-6&	1.59e-3&	7.19e-4&	1.98e-3&	8.98e-4&	1.03e-2&	2.73e-3 \\
 &8&	2.60e-7&	1.47e-8&	1.61e-5&	3.13e-6&	2.01e-5&	3.91e-6&	2.14e-4&	7.46e-5 \\
3&16&	3.86e-9&	5.89e-11&		2.09e-7&	1.22e-8&	2.61e-7&	1.53e-8&	2.32e-6&	2.82e-7 \\
 &32&	5.95e-11&		2.23e-13&	3.10e-9&	4.81e-11&		3.86e-9&	6.00e-11&		3.24e-8&	1.09e-9 \\
 & $\rho_3$	& 6.10&	7.96&	6.32&	7.95&	6.32&	7.95&	6.13&	7.18 \\ \hline
 
\end{tabular}
\caption{Relative eigenvalue errors of mixed isogeometric analysis using standard Gauss (G) quadratures and optimally-blended (O) rules for biharmonic eigenvalue problem.}
\label{tab:bhev} 

\vspace{0.2cm}
\begin{tabular}{| c | c | cc | cc | cc | cc |}
\hline
$p$ & $N$ & \multicolumn{2}{c|}{$e_{\lambda_1}$} & \multicolumn{2}{c|}{$e_{\lambda_2}$} & \multicolumn{2}{c|}{$e_{\lambda_4}$} & \multicolumn{2}{c|}{$e_{\lambda_8}$} \\
 & & G &  O & G &  O & G &  O & G &  O  \\[0.1cm] \hline

  &4	&1.05e-1&	3.16e-3&	3.96e-1&	4.34e-2&	4.75e-1&	5.36e-2&	8.51e-1&	1.99e-1 \\
  &8	&2.54e-2&	1.95e-4&	9.00e-2&	2.61e-3&	1.07e-1&	3.22e-3&	2.07e-1&	1.26e-2 \\
1&16	&6.29e-3&	1.21e-5&	2.19e-2&	1.60e-4&	2.58e-2&	1.98e-4&	4.88e-2&	7.62e-4 \\
  &32	&1.57e-3&	7.56e-7&	5.42e-3&	9.98e-6&	6.40e-3&	1.23e-5&	1.20e-2&	4.72e-5 \\
  & $\rho_1$	& 2.02&	4.01&	2.06&	4.03&	2.07&	4.03&	2.05&	4.02 \\[0.1cm] \hline

  &4	&1.17e-3&	8.47e-5&	2.13e-2&	4.40e-3&	2.64e-2&	5.51e-3&	1.34e-1&	2.99e-2 \\
  &8	&6.66e-5&	1.31e-6&	9.64e-4&	6.90e-5&	1.19e-3&	8.62e-5&	5.21e-3&	7.15e-4 \\
2&16	&4.06e-6&	2.04e-8&	5.49e-5&	1.07e-6&	6.78e-5&	1.33e-6&	2.69e-4&	1.10e-5 \\
  &32	&2.52e-7&	3.18e-10&	3.35e-6&	1.66e-8&	4.14e-6&	2.07e-8&	1.60e-5&	1.71e-7 \\
  & $\rho_2$	& 4.06&	6.01&	4.20&	6.01&	4.21&	6.01&	4.34&	5.83 \\[0.1cm] \hline

  &4	&1.90e-5&	3.35e-6&	1.57e-3&	7.12e-4&	1.97e-3&	8.92e-4&	1.02e-2&	2.72e-3 \\
  &8	&2.54e-7&	1.43e-8&	1.60e-5&	3.10e-6&	2.00e-5&	3.89e-6&	2.13e-4&	7.43e-5 \\
3&16	&3.77e-9&	5.75e-11&		2.07e-7&	1.21e-8&	2.59e-7&	1.52e-8&	2.31e-6&	2.80e-7 \\
  &32	&5.81e-11&		2.02e-13&	3.07e-9&	4.76e-11&		3.84e-9&	5.96e-11&		3.22e-8&	1.09e-9 \\
  & $\rho_3$	& 6.10&	7.99&	6.32&	7.95&	6.32&	7.95&	6.13&	7.18 \\ \hline

\end{tabular}
\caption{Relative eigenvalue errors of mixed isogeometric analysis using standard Gauss (G) quadratures and optimally-blended (O) rules for Cahn-Hilliard eigenvalue problem.}
\label{tab:chev} 
\end{table}

\begin{table}[h!]
\centering 
\begin{tabular}{| c | c | cc | cc | cc | cc |}
\hline
$p$ & $N$ & \multicolumn{2}{c|}{$e_{\lambda_1}$} & \multicolumn{2}{c|}{$e_{\lambda_2}$} & \multicolumn{2}{c|}{$e_{\lambda_4}$} & \multicolumn{2}{c|}{$e_{\lambda_8}$} \\
 & & G &  O & G &  O & G &  O & G &  O  \\[0.1cm] \hline

  &4	&1.13e-1&	3.42e-3&	4.09e-1&	4.48e-2&	4.85e-1&	5.46e-2&	8.63e-1&	2.02e-1 \\
  &8	&2.74e-2&	2.10e-4&	9.29e-2&	2.69e-3&	1.09e-1&	3.28e-3&	2.10e-1&	1.27e-2 \\
1&16	&6.79e-3&	1.31e-5&	2.25e-2&	1.65e-4&	2.63e-2&	2.02e-4&	4.94e-2&	7.71e-4 \\
  &32	&1.69e-3&	8.16e-7&	5.59e-3&	1.03e-5&	6.53e-3&	1.26e-5&	1.21e-2&	4.77e-5 \\
  & $\rho_1$	& 2.02&	4.01&	2.06&	4.03&	2.07&	4.03&	2.05&	4.02 \\[0.1cm] \hline

  &4	&1.26e-3&	9.14e-5&	2.19e-2&	4.54e-3&	2.69e-2&	5.61e-3&	1.35e-1&	3.02e-2 \\
  &8	&7.19e-5&	1.41e-6&	9.94e-4&	7.11e-5&		1.22e-3&	8.79e-5&	5.27e-3&	7.24e-4 \\
2&16	&4.39e-6&	2.20e-8&	5.66e-5&	1.10e-6&	6.91e-5&	1.36e-6&	2.72e-4&	1.11e-5 \\
  &32	&2.72e-7&	3.43e-10&	3.45e-6&	1.71e-8&	4.22e-6&	2.11e-8&		1.61e-5&	1.73e-7 \\
  & $\rho_2$	& 4.06&	6.01&	4.20&	6.01&	4.21&	6.01&	4.34&	5.83 \\[0.1cm] \hline

  &4	&2.05e-5&	3.62e-6&	1.62e-3&	7.34e-4&	2.01e-3&	9.09e-4&	1.03e-2&	2.76e-3 \\
  &8	&2.74e-7&	1.54e-8&	1.65e-5&	3.20e-6&	2.04e-5&	3.96e-6&	2.15e-4&	7.52e-5 \\
3&16	&4.07e-9&	6.21e-11&		2.14e-7&	1.25e-8&	2.64e-7&	1.55e-8&	2.34e-6&	2.84e-7 \\
  &32	&6.26e-11&		1.76e-13&	3.17e-9&	4.90e-11&		3.91e-9&	6.08e-11&		3.26e-8&	1.10e-9 \\
  & $\rho_3$	& 6.10&	8.08&	6.32&	7.95&	6.32&	7.95&	6.13&	7.18 \\[0.1cm] \hline

\end{tabular}
\caption{Relative eigenvalue errors of mixed isogeometric analysis using standard Gauss (G) quadratures and optimally-blended (O) rules for Swift-Hohenberg eigenvalue problem.}
\label{tab:shev}

\vspace{0.2cm}
\begin{tabular}{| c | c | cc | cc | cc | cc |}
\hline
$p$ & $N$ & \multicolumn{2}{c|}{$e_{\lambda_1}$} & \multicolumn{2}{c|}{$e_{\lambda_2}$} & \multicolumn{2}{c|}{$e_{\lambda_4}$} & \multicolumn{2}{c|}{$e_{\lambda_8}$} \\
 & & G &  O & G &  O & G &  O & G &  O  \\[0.1cm] \hline
 & 4&	1.08e-1&	3.24e-3&	4.00e-1&	4.39e-2&	4.78e-1&	5.39e-2&	8.55e-1&	2.00e-1 \\
 
  &4	&1.72e-1&	5.03e-3&	6.67e-1&	6.59e-2&	8.06e-1&	8.04e-2&	1.54e0&	2.86e-1 \\
  &8	&4.07e-2&	3.10e-4&	1.42e-1&	4.00e-3&	1.67e-1&	4.90e-3&	3.30e-1&	1.90e-2 \\
1&16	&1.00e-2&	1.93e-5&	3.38e-2&	2.46e-4&	3.96e-2&	3.02e-4&	7.48e-2&	1.15e-3 \\
  &32	&2.50e-3&	1.20e-6&	8.34e-3&	1.53e-5&	9.76e-3&	1.88e-5&	1.82e-2&	7.14e-5 \\
  & $\rho_1$	& 2.03&	4.01&	2.10&	4.02&	2.12&	4.02&	2.13&	3.99 \\[0.1cm] \hline
  
  &4	&1.86e-3&	1.35e-4&	3.29e-2&	6.77e-3&	4.05e-2&	8.39e-3&	2.09e-1&	4.55e-2 \\
  &8	&1.06e-4&	2.08e-6&	1.48e-3&	1.06e-4&	1.82e-3&	1.31e-4&	7.90e-3&	1.08e-3 \\
2&16	&6.47e-6&	3.24e-8&	8.43e-5&	1.64e-6&	1.03e-4&	2.03e-6&	4.07e-4&	1.66e-5 \\
  &32	&4.02e-7&	5.06e-10&	5.14e-6&	2.55e-8&	6.30e-6&	3.16e-8&	2.42e-5&	2.58e-7 \\
  & $\rho_2$	& 4.06&	6.01&	4.21&	6.01&	4.21&	6.01&	4.35&	5.83 \\[0.1cm] \hline

  &4	&3.02e-5&	5.34e-6&	2.42e-3&	1.09e-3&	3.00e-3&	1.36e-3&	1.55e-2&	4.12e-3 \\
  &8	&4.04e-7&	2.28e-8&	2.45e-5&	4.76e-6&	3.04e-5&	5.92e-6&	3.22e-4&	1.12e-4 \\
3&16	&6.00e-9&	9.12e-11&		3.18e-7&	1.86e-8&	3.94e-7&	2.31e-8&	3.50e-6&	4.24e-7 \\
  &32	&9.09e-11&		9.10e-13&	4.71e-9&	7.23e-11&		5.84e-9&	9.09e-11&		4.88e-8&	1.66e-9 \\
  & $\rho_3$	& 6.11&	7.54&	6.32&	7.95&	6.32&	7.95&	6.14&	7.17 \\[0.1cm] \hline
    
\end{tabular}
\caption{Relative eigenvalue errors of mixed isogeometric analysis using standard Gauss (G) quadratures and optimally-blended (O) rules for phase-field crystal eigenvalue problem.}
\label{tab:pfcev} 
\end{table}

Figure \ref{fig:fig7} shows the convergence of the eigenvalues $\lambda_j, j=1,2,4,8$ of the 2D biharmonic problem. The squared eigenvalue errors of the quadratic mixed isogeometric analysis with standard Gauss quadratures have convergence order close to 4 (that is $2p$). Using optimally-blended rules leads to smaller eigenvalue errors and convergence order of 6 (that is $2p+2$). Figure \ref{fig:fig8} shows the convergence of 2D Cahn-Hilliard operator eigenvalues when periodic boundary conditions are employed. Similar results are obtained for the phase-field equations.

\begin{figure}[!ht]
\centering\includegraphics[width=13cm]{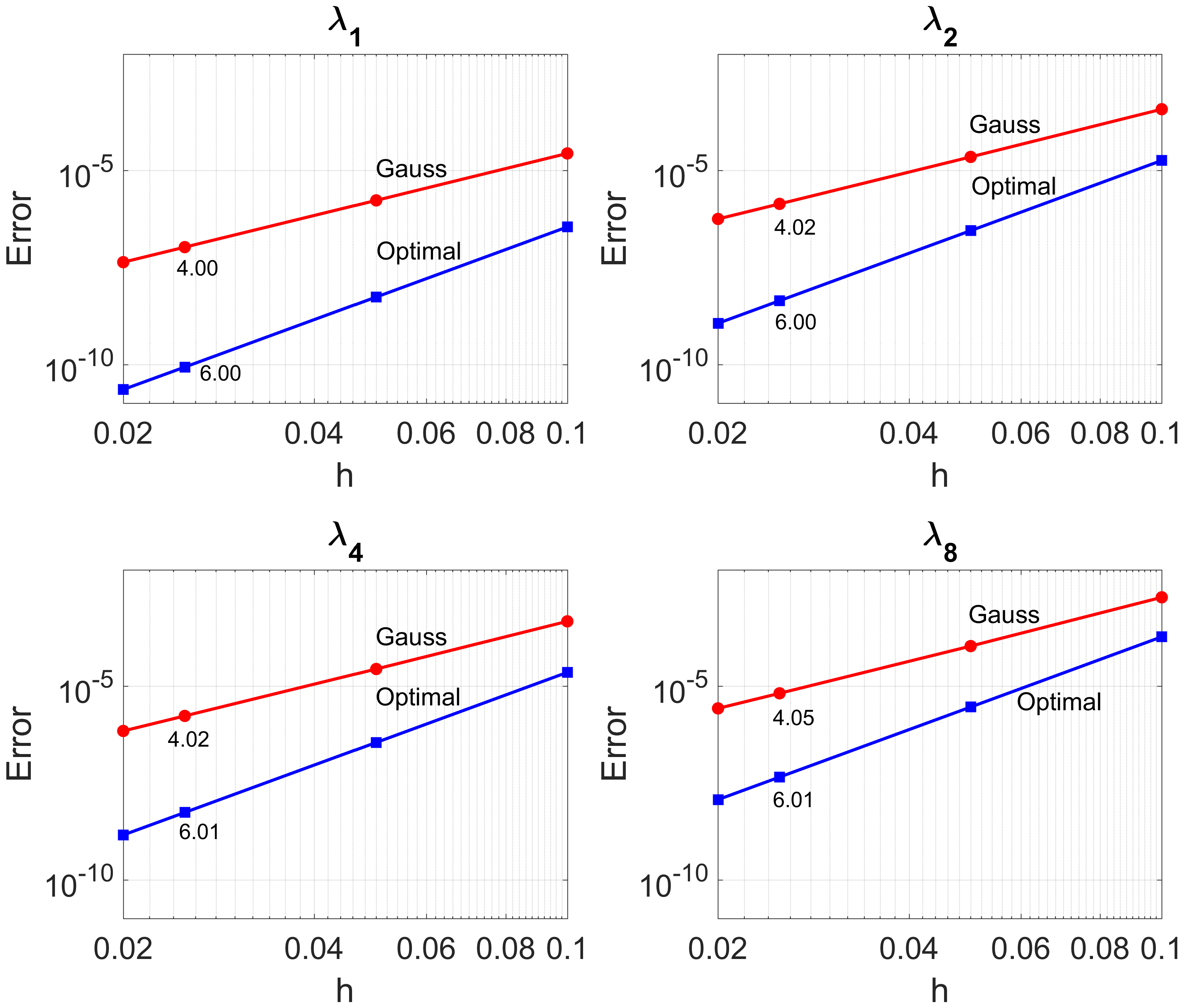}
\caption{The convergence of the eigenvalues $\lambda_j, j=1,2,4,8$ of the 2D biharmonic problem (simply supported plate).}
\label{fig:fig7}
\end{figure}

\begin{figure}[!ht]
\centering\includegraphics[width=13cm]{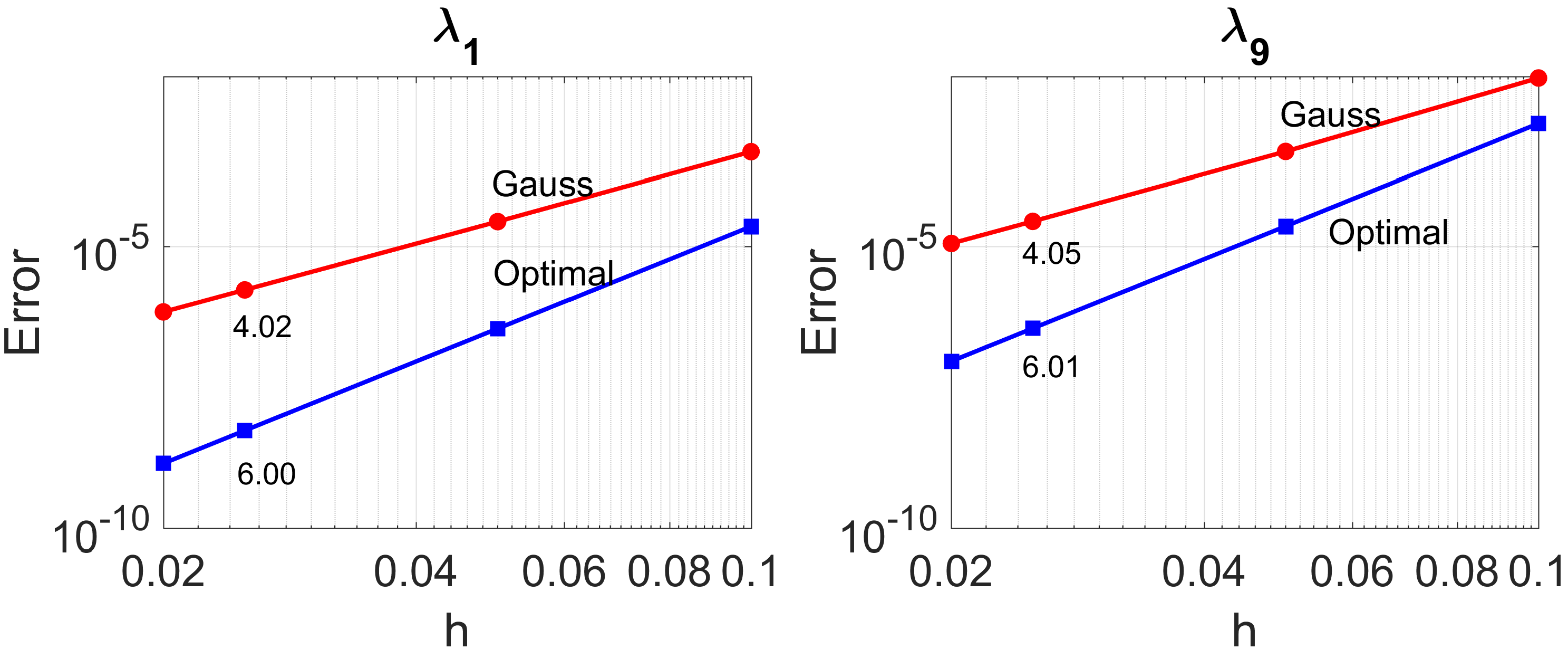}
\caption{The convergence of the eigenvalues $\lambda_1$ and $\lambda_{9}$ of the 2D Cahn-Hilliard equation with periodic boundary conditions.}
\label{fig:fig8}
\end{figure}

\section{Conclusions and future outlook} \label{sec:con}
We present and study a mixed formulation of isogeometric analysis for a set of $2n$-order differential eigenvalue problems, which includes operators arising from the biharmonic problem, Cahn-Hilliard, Swift-Hohenberg, and phase-field crystal equations. We show that the mixed formulation with standard quadrature rules applied to integrate the inner-products leads to optimal rates ($2p$) of convergence on eigenvalue errors while the mixed formulation with optimally-blended rules leads to super-convergent $2p+2$ approximated eigenvalues. This work generalizes the theoretical results obtained for Laplacian eigenvalue problem to higher-order differential eigenvalue problems.

Future work in this direction includes generalization to other higher-order differential eigenvalue problem as well as to nonlinear differential eigenvalue problems. We expect the isogeometric discretization of nonlinear differential eigenvalue problems leads to nonlinear matrix eigenvalue problem, which requires advanced numerical solvers. Developing both optimally-blended quadrature rules for nonlinear differential eigenvalue problems and fast and efficient numerical solvers for their corresponding nonlinear matrix eigenvalue problems are subject to future work.

\section*{Acknowledgement}
This publication was made possible in part by the CSIRO Professorial Chair in Computational Geoscience at Curtin University and the Deep Earth Imaging Enterprise Future Science Platforms of the Commonwealth Scientific Industrial Research Organisation, CSIRO, of Australia. Additional support was provided by the European Union's Horizon 2020 Research and Innovation Program of the Marie Sk{\l}odowska-Curie grant agreement No. 777778, the Mega-grant of the Russian Federation Government (N 14.Y26.31.0013), the Institute for Geoscience Research (TIGeR), and the Curtin Institute for Computation. The J. Tinsley Oden Faculty Fellowship Research Program at the Institute for Computational Engineering and Sciences (ICES) of the University of Texas at Austin has partially supported the visits of VMC to ICES.





\end{document}